%% file: Shafarevich.tex
\documentclass[12pt]{amsart}

\input{pgmacs.tex}

\usepackage{tikz-cd}

\usepackage{amsmath}
   \usepackage{esint}
  \usepackage[matrix,arrow]{xy}
\usepackage{xcolor}
\usepackage{tikz}
\usetikzlibrary{patterns}
\usetikzlibrary{decorations.pathreplacing}
\usetikzlibrary{arrows}
\usepackage{stackrel}
\usepackage{tabstackengine}
\stackMath
\makeatletter
\renewcommand\TAB@delim[1]{#1}

\makeatother

  \usepackage{fullpage}  
\usepackage{enumerate}
    \usepackage{amssymb}
 \usepackage[mathscr]{euscript}
  \newcommand\bsp[1]{\begin{split} #1 \end{split}}
  
  \newcommand\beb{\begin{enumerate}[$\bullet$]}
 \newcommand\eeb{\end{enumerate}}
  
    \newcommand\Gy{\mathop{\rm Gy}\nolimits}
        \newcommand\Nil{\mathop{\rm Nil}\nolimits}

\newcommand\biti[1]{\bibitem[#1]{#1}}

     \newcommand\ssni[1]{\medbreak\noindent {\it #1\/}.}

  \theoremstyle{mytheo}

 \theoremstyle{note}

   \renewcommand\cG{\mathcal{G}}
   \newcommand\Shaf{Shafarevich}
\newcommand\Sh{\mathop{\text{\it Sh}}\nolimits}
\newcommand\wtsh{\wt{\Sh}}

\newcounter{Cequ}
\newenvironment{CEquation}
  {\stepcounter{Cequ}%
    \equation}
  {\endequation}
      \begin{document}
      
      \title{Shafarevich mappings and period mappings}
      \author{Mark Green, Phillip Griffiths and Ludmil Katzarkov}
      \begin{abstract}
      We shall show that a smooth, quasi-projective variety $X$ has a holomorphically convex universal covering $\wt X$ when (i) $\pi_1(X)$ is residually nilpotent and (ii) there is an admissable variation of \mhs\ over $X$ whose monodromy representation has a finite kernel, and where in each case a corresponding period mapping is assumed to be proper.
      \end{abstract}
      \dedicatory{To our friend Enrico Arbarello, whose zest for life and enjoyment of mathematics continue to inspire us} 
      
 \address{Department of Mathematics, University of California at Los Angeles,\hfill\break\indent Los Angeles, CA 90095\hfill\break\indent
 {\it E-mail address}\/: {\rm mlg@ipam.ucla.edu}\hfill\break\indent  Institute for Advanced Study, Einstein Drive, Princeton, NJ 08540\hfill\break\indent
{\it E-mail address}\/: {\rm pg@ias.edu}\hfill\break\indent
University of Miami, Coral Gables, FL
{\rm and} Institute of Mathematics\hfill\break\indent and Informatics, Bulgarian Academy of Sciences
 NRU HSE, Moscow\hfill\break\indent
{\it E-mail address}\/: {\rm lkatzarkov@gmail.com}
} 
      \maketitle

      \section*{Outline}
      \begin{enumerate}[I.]
      \item \emph{Introduction and statements of results}
      \item \emph{The nilpotent   case}
      \item \emph{The variation of Hodge structure and mixed Hodge structure cases}
      \item \emph{Further directions}
      \end{enumerate}
      
      \section{Introduction and statements of results}\lab{sec1}
     
     In \cite{Sh72} \Shaf\ posed a beautiful question, a variant of which has become known as Shafarevich's conjecture (cf. \cite{Ko93} and  \cite{Ko95}):
     \begin{subsec}\lab{1.1}
     \emph{Let $X$ be a smooth projective variety.  Is the universal covering $\wt X$ of $X$ holomorphically convex?}\end{subsec}
     This means that there is a Stein analytic variety $\wt S$ and a proper holomorphic mapping $\wt X\to \wt S$ that contracts the connected compact analytic subvarieties of $\wt X$ to points.  This mapping is the Cartan-Remmert reduction of $\wt X$ (\cite{Car79}).
     
     In practice for a quasi-projective variety $X$  one seeks a \emph{\Shaf\ mapping} (\cite{Ko93}, \cite{Ko95} and \cite{Cam94})
     \begin{equation}\lab{1.2}
     \Sh: X\to S\end{equation}
     onto an analytic variety with the properties 
     \begin{enumerate}[(i)] \item If $Y\subset X$ is a connected complex analytic subvariety of $X$ that is contracted to a point by $\Sh$, then $Y$ is compact and the image 
     \[
     \pi_1(Y)\to \pi_1(X)\]
     is a finite subgroup (we shall say that the mapping $\pi_1(Y)\to \pi_1(X)$ is finite).
     \item In the diagram
    \vspace*{-4pt} \begin{equation}\lab{1.3}\bsp{
     \xymatrix@R=1.5pc{\wt X\ar[d]_\pi \ar[r]^{\wt {\Sh}}&\wt S\ar[d]\\
     X\ar[r]^{\Sh}&S}}\vspace*{-4pt}\end{equation}
     the universal cover  $\wt S=\wt X\times_{_{\!X}} S$ is Stein.

      \end{enumerate}
     
     We note that if (i) is satisfied, then the connected fibres of $\wtsh$
     are the connected components of $\pi^{-1}(Y)$ where $Y\subset X$ is a connected fibre of $\Sh$.  Moreover the components of $\pi^{-1}(Y)$ are compact.
     
     The \Shaf\ conjecture has been established when 
     \begin{enumerate}[(a)]
     \item $X$ is projective and $\pi_1(X)$ is residually nilpotent;
     \item $X$ is projective and $\pi_1(X)$ has a faithful linear representation.
     
  \end{enumerate}
  
  The initial  proof of (a) is in \cite{Ka97}.  Subsequently there have been numerous further works (cf.\ \cite{Cam95} and \cite{Cl08}).     The proof of (b) is in \cite{Ey04} and  \cite{EKPR12}; it is the culmination of a series of works drawing on both classical and non-abelian Hodge theory (cf.\ \cite{CM-SP17} and \cite{Si88}, \cite{Si92}).  Some of the history and references to the literature are given in \cite{EKPR12}.  Further related references are \cite{BdO06}, \cite{EF21}, \cite{KR98}, and \cite{Ey09}.

  The purpose of this partly expository paper is to   extend  (a) and a special case of (b) to the case when $X$ is quasi-projective.  The essence of what will be proved here may be informally expressed as saying
     \[     \text{\emph{proper period mappings give \Shaf\ mappings}.}\]
     
In order to understand what is meant by proper for the period mappings to be considered we  introduce some standard terminology and notations.  By a \emph{completion} of $X$ we mean a smooth projective variety $X$ in which $\ol X$ is a Zariski open set with complement $Z=\cup Z_i$ a reduced normal crossing divisor  with $Z_i$ irreducible.
     
     In case (a) there are two choices for a period mapping.  One is the classical Albanese mapping
     \begin{subequations}
     \begin{equation}
     \lab{1.4a}
     \alpha:X\to\Alb(X)\end{equation}
     where $\Alb(X)=H^0(\Om^1_X(\log Z))^\ast/H_1(X,\Z)$ is a semi-abelian variety.  The other choice is to use a higher Albanese mapping (\cite{Ha87a} and \cite{Cl08})
     \begin{equation}
     \lab{1.4b}
     \alpha_s:X\to\Alb^s(X), \quad s=1,2,3,\dots .\end{equation}\end{subequations}
This is a period mapping for a unipotent variation of \mhs\ (UVMHS) (\cite{HZ85}).  It reduces to \eqref{1.4a} when $s=1$, and when $\pi_1(X)$ is residually nilpotent the mappings \eqref{1.4b} are all essentially the same for $s\gg 0$.

The properness of the Albanese mapping \eqref{1.4a} implies that the image of the residue mapping
\[
H^0(\Om^1_{\ol X}(\log Z))\to\opplus^i H^0(\cO_{Z_i}) (-1)\]
projects non-trivially to each summand $H^0(\cO_{Z_i})(-1)$.  These conditions are independent of the completion of $X$.  Additional for each stratum $Z_I:= \cap_{i\in I}Z_i$ with $|I| \geqq 2$ and each ray $\la$ emanating from 
$Z_I^\ast$ out into $X$ there should be an $\om\in H^0(\Om^1_{\ol X}(\log Z)$ having a logarithmic singularity at $\la\cap Z^\ast_I$.  If the second of these conditions is not satisfied, then only the closure of the image of the \Shaf\ map will be Stein.  
From the second of the three proofs of  Theorem A below it will follow that $\alpha$ is proper if, and only if, the $\alpha_s$ are proper.  In fact as a consequence of the proof of Theorem \ref{2.4} below we will  have the 
\begin{Thm} \lab{1.5}
 The images $\alpha_s(X)$ are algebraic varieties and the mapping $\alpha_s(X)\to\alpha(X)$ is a proper surjective morphism.  If $Y\subset X$ is a compact subvariety with $\alpha\big|_Y=$ constant, then $\alpha_s\big|_Y=$ constant for $s\gg 0$.
 \end{Thm}

To state the result in case (a) we will use the diagram
\begin{equation}\lab{1.6}
\bsp{
\xymatrix{\wt X\ar[d]_{\pi}\ar[r]^{\wt\alpha}&\wt{\Alb(X)}\ar[d]\\
X\ar[r]^{\alpha}&\Alb(X)}}\end{equation}
where $\wt {\Alb(X)}$ is the universal cover $\Alb(X)$, together with similar diagrams (I.6$_s$) for the higher Albanese mappings.

\begin{thma}
Assume that $\pi_1(X)$ is residually nilpotent and that $\alpha$ is proper.  
\begin{enumerate}[{\rm (i)}]
\item  The mapping $\wt\alpha$ is proper;
\item the image $\wt \alpha(\wt X)$ is Stein;
\item the connected components  of the fibres of $\wt \alpha$ are the connected components of $\pi^{-1}(Y)$ where $Y\subset X$ is a fibre of $\alpha$;
\item the connected components of the fibres of $\alpha$ are characterized by the condition that the map
\[
H^1(X,\Q) \to H^1(Y,\Q)\]
be trivial.\end{enumerate}
\end{thma}
The substantive parts of this theorem are (ii), (iii) and (iv).  We note that (iii) and (iv) imply
\begin{equation}\lab{1.7}
\pi_1(Y)\to \pi_1(X)\text{ is finite }\iff H^1(X,\Q)\to H^1(Y,\Q)  \text{ is trivial.}\end{equation}
The implication $\Rightarrow$ is clear.  The proof of the converse will use Hodge theory.  One argument will be an extension of that used in \cite{Ka97}.  A second argument, one that is more in the conceptual framework of general period mappings,  will use that the implication $\Leftarrow$ is a consequence of  Theorem \ref{1.5}.    A third argument,  different from the first two, will  describe a    construction of the  minimal model of  $\wh{\pi_1}(X)$ as in \cite{Mo78} and \cite{Ha87a}.  For use elsewhere this construction will give a bound on  the singularities along $Z$ of the differential forms in the minimal model.  We will also give an extension to the quasi-projective case of the classical $\part\ol\part$-lemma.

For case (b) we shall first give a result where we have a standard variation of Hodge structure (VHS) whose underlying local system $\V\to X$ has a monodromy representation
\[
\rho:\pi_1(X)\to \Ga\subset \Aut(V)\]
with a finite kernel (\cite{CM-SP17}).  We denote by
\begin{equation}\lab{1.8}
\Phi:X\to\Ga\bsl D\end{equation}
the corresponding period mapping (loc.\ cit.).  The properness of $\Phi$ is equivalent to the logarithm  $N_i$ of monodromy around each  $Z_i$ being non-zero.  Again this condition is independent of the smooth completion of $X$.  In this case it is well known that  the image
\[
\Phi(X):=P\subset \Ga\bsl D\]
is a closed analytic subvariety of $\Ga\bsl D$.\footnote{It is in fact a quasi-projective algebraic subvariety (\cite{BBT06}).}   We then have a diagram
\begin{equation}\lab{1.9}
\bsp{\xymatrix{\wt X\ar[r]^{\wt\Phi}\ar[d]_\pi& \wt P\subset D\ar[d]\\
X\ar[r]^{\Phi}& P\subset \Ga\bsl D}}\end{equation}
where $\wt P$ is the image in $D$ of $\wt X$.
\begin{thmb}
Assume that $\Phi$ is proper and that $\rho$ has a finite kernel.
\begin{enumerate}[{\rm (i)}]
\item $\wt\Phi$ is a proper holomorphic mapping;
\item $\wt P$ is a Stein subvariety of $D$;
\item the connected fibres of $\wt \Phi$ are the connected components of $\pi^{-1}(Y)$ where $Y\subset X$ is a fibre of $\Phi$;
\item these $Y$ are characterized by the condition that $\rho\big|_{\pi_1(Y)}$ is finite.\end{enumerate}
\end{thmb}

Given our assumptions the main substantive statement in the theorem is (ii).  As will be seen in Section III this will be a direct consequence of classical results about the geometry of period domains.

Theorem B may be summarized by saying that $\Phi$ is a \Shaf\ mapping for $X$.  What is missing is a cohomological, rather than a homotopy-theoretic, description of the fibres of $\wt\Phi$.  One would like some analogue of (iv) in Theorem A.  This will be further discussed in Section III.

Our final main result is an amalgam of Theorems A and B.  To state it we assume given an admissible variation of \mhs\ (VMHS) with the local system $V_M\to X$ having corresponding monodromy representation $\rho_M:\pi_1(X)\to\Ga\subset \Aut(V_M)$ (cf.\ \cite{SZ85} and \cite{Us83}).  We denote by
\begin{equation}\lab{1.10}
\Phi_M:X\to\Ga_M\bsl D_M\end{equation}
the corresponding period mapping where $D_M$ is now a mixed period domain (loc.\ cit.).

\begin{thmc}
Assume that $\Phi_M$ is proper and that $\rho_M$ has a finite kernel.  Then $\Phi_M$ is a \Shaf\ map for $X$.\end{thmc}

The obvious  idea behind the proof is to consider the factorization
\[
\xymatrix{X\ar[r]^{\Phi_M\quad}\ar[dr]_\Phi&\Ga_M\bsl D_M\ar[d]\\
&\Ga\bsl D}\]
where $\Phi$ is the   period mapping obtained by passing to the associated graded of the \mhs s parametrized by $D_M$.  For this period mapping we have a diagram \eqref{1.9} where the image $\wt \Phi(\wt X)\subset D$ will be a Stein analytic variety.  Note that we are not saying that $\wt \Phi$ is proper.

Along the fibres of $\Phi$ we will have an admissible VMHS such that the associated graded polarized Hodge structures  are locally constant on the irreducible components  of a fibre $Y$.\footnote{This is the situation that one encounters in the Satake-Baily-Borel completion of a classical period mapping (\cite{GGLR20}, \cite{GGR22} and \cite{Gr22}).  In these references a more general version of Theorem 1.5 appears; also given there are further details of some arguments that are only sketched here.}  Then a strengthening of \eqref{1.5} will give that the image $\Phi_M(Y)$ of the fibre will map  finitely to the image under $\Phi_M$ of the extension data of levels $\leqq 2$ (this condition will be explained below).  The proof of Theorem A can be adapted to this situation to show $\wt \Phi(\pi^{-1}(Y))$ is Stein.  We will then have an analytic variety mapping to a Stein analytic variety with Stein fibres.  In the situation at hand the total space can then be shown to be a Stein variety (cf.\ \cite{St56}).

Following the proof in Section III of Theorems B  and C we will discuss the question of describing cohomologically  the fibres of the \Shaf\ mapping in case B.  This leads naturally into the description of the \mhs\ on the completion relative to $\rho$ of $\pi_1(X,x)$ (\cite{Ha98} and \cite{EKPR12}).  The cohomology groups that enter here have one  interpretation  as defining the extension classes of that \mhs, and among these are the classes of first order deformations of the local system $\V\to X$ underlying a VHS.  Both of these are central ingredients in the proof of the main theorem in \cite{EKPR12}.  In \cite{Le19} and \cite{Le21} these results have been extended to the quasi-projective case. 
Here for expository purposes we will discuss very briefly  special cases of  the results in  these references.  

In  Section IV we will discuss some questions and conjectures related to different perspectives on the \Shaf\ conjecture.

\subsection*{Acknowledgements}
We are very grateful  to P.  Eyssidieux, T. Pantev and C. Simpson for useful discussions.

The third  author was partially supported by NSF Grant, Simons
Investigator Award HMS, Simons.
Collaboration Award HMS, National Science Fund of Bulgaria, National Scientific Program ``Excellent Research and People for the Development of European Science'' (VIHREN), Project No. KP-06-DV-7, HSE University Basic
Research Program.

\section{The nilpotent  case} \setcounter{equation}{0}
Referring to the diagram \eqref{1.6} and statement of Theorem A there are two things to be proved:
\begin{align}
\lab{2.1}
\bmp{5.5}{if $Y\subset X$ is a connected fibre of $\alpha$, then the map $\pi_1(Y)\to\pi_1(X)$ is finite;}\\
\bmp{5.5}{there is a plurisubharmonic (psh) exhaustion function $\vp:\wt\alpha(\wt X)\to \R$.}\lab{2.2}\end{align}

For \eqref{2.1} we will give two different arguments.  The first is basically to extend the proof in \cite{Ka97} to the quasi-projective case (cf.\ also \cite{EKPR12} and \cite{Cl08}).  The steps are
\beb
\item reduce to the case where $Y\subset X$ has dimension 1;
\item further reduce to the case where $Y$ has nodes.
\eeb
These steps are essentially the same as in \cite{Ka97} and that are reviewed in Section 3 of \cite{EKPR12}.
\beb
\item analyze the case when $Y$ is a cycle; i.e., the dual graph of $Y$ is topologically an $S^1$.
\[\begin{picture}(40,60)
\put(0,0){\includegraphics{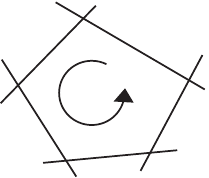}}
\put(22,22){$\ga$}\end{picture}  \]
\eeb
In this case what \eqref{2.1} means is that the inverse image $\pi^{-1}(Y)\subset \wt X$ is not an infinite chain of irreducible curves.  Equivalently, some multiple of the circuit  $\begin{picture}(20,20)\put(0,-5){\includegraphics{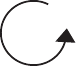}}\put(5,4){$\ga$}\end{picture}$ maps to the identity in $\pi_1(X)$.  The general case of a nodal curve can be done by extending the argument in this case.

Our assumption is that the map $H_1(Y,\Q)\to H_1(X,\Q)$ is trivial.  We want to use this statement about homology to infer one about homotopy.  For this we shall use that the unipotent completions
\[
\widehat{\Z \pi_1}(X;x)=\lim_{s\to\infty} \pi_1(X,x)/J^{s+1}\]
of $\pi_1(X,x)$ and similarly for $Y$ have \mhs s (\cite{Mo78} and \cite{Ha87b}).  Here $\Z {\pi_1}(X,x)$ is the group ring of $\pi_1(X,x)$ and $J$ is the augmentation ideal.  The weights of the generators coming from $H_1(X)$ and $H_1(Y)$ are
\beb
\item $\widehat{\pi_1(X)}$ is generated in weights $-1,-2$ (cf.\ \cite{Mo78});
\item $\widehat{\pi_1(Y)}$ is generated in weights $-1,0$ (cf.\ \cite{CG14}).
\eeb
If $Y=\cup Y_\alpha$, the -1 part comes from the $H_1(Y_\alpha,\Q)$.  The $0$-part, which is only well defined modulo the $-1$ parts, corresponds to the circuit $\ga$.  This may be seen by considering the exact sequence of the pair $(Y,D)$ 
\[
 H_1(Y)\to H_1(Y,D)\xri{\part}H_0(D) \]
where $D$ is the set of nodes.   Here $\ga$ should be considered as a class in $\ker\part$ modulo the images of the $H_1(Y_\alpha)$.  From this weight considerations give that the class of $\ga$ maps to zero in $\wh{\pi_1(X)}$.
   Since $\pi_1(X)$ is assumed to be residually nilpotent we may   conclude that $\ga$ is of finite order in $\pi_1(X)$.\hfill\qed

\medbreak

We will now give the second argument for the proof of \eqref{2.1}.  This argument is more conceptual.  It takes place within the general framework of period mappings and will be used again in the proof of Theorem C.  We begin by recalling from \eqref{1.4b} the higher Albanese mappings constructed in \cite{Ha87a} and \cite{HZ87}.  For each $s\geqq 1$ there is a diagram of mappings
\begin{equation}\lab{2.3}
\bsp{
\xymatrix{&\Alb^s(X)\ar[dd]^{\pi_s}\\
X\ar[ur]_{\alpha_s}\ar[dr]_\alpha&\\
&\Alb(X)}}
\end{equation}
where $\Alb(X)=\Alb^1(X)$ and $\alpha=\alpha_1$ is the usual Albanese mapping.  These maps may be viewed as period mappings associated to unipotent variations of \mhs\   whose underlying local systems are the groups
\[
\V^s_x = \Z \pi_1(X,x)/J^{s+1}.\]
A basic property of the induced mappings in \eqref{2.3} is given by Theorem \ref{1.5} above.  Here   we only need the following part of that result.
 \begin{Thm}\lab{2.4}
For $Y\subset X$ a compact subvariety and $s\gg 0$
\[
\alpha\big|_Y=\text{constant }\iff \alpha_s\big|_Y=\text{constant}.\]
\end{Thm}

The proof of Theorem \ref{2.4} will be a consequence of  a general result about the variation of the extension data in a VMHS that we will now explain and give a sketch of the proof.

Let 
\begin{equation}\lab{2.5}
\Phi:X\to\Ga\bsl D\end{equation}
be any admissible variation of \mhs s whose associated graded is a constant family $\{H^0,\dots,H^n\}$ of weight $k$ Hodge structures $H^k$ on a fixed vector space $\Gr^W_kV$.  What is then varying is the extension data for the \mhs s.  This extension data comes in various levels (cf. \ref{2.8} below),   
 and associated to \eqref{2.5} there are maps
\begin{subequations}
\begin{align}
X &\xri{\Phi_1} \{\text{extension data of level 1}\},\lab{2.6a}\\
\{\text{fibre of }\Phi_1\} &\xri{\Phi_{1,2}} \{\text{extension data of level 2}\},\lab{2.6b}\\
\{\text{fibre of }\Phi_{2}\}&\xri{\Phi_{2,3}} \{\text{extension data of level 3}\}.\lab{2.6c}\\
&\qquad\qquad\vdots\nonumber\end{align}\end{subequations}
The right-hand sides of the above maps are isomorphic to 
\begin{align*}
\text{level 1}&\quad \opplus^\ell \Ext^1_{\MHS} (H^{\ell+1},H^\ell),\\
\text{level 2}&\quad \opplus^\ell \Ext^1_{\MHS}(H^{\ell+2},H^\ell),\\
\text{level 3}&\quad \opplus^\ell \Ext^1_{\MHS}(H^{\ell+3},H^\ell)\\
&\qquad\qquad \vdots\end{align*}
This means that if we pick a point in a fibre, then the difference of the extension data for a variable point and the fixed point is in the $\oplus \Ext^1_{\MHS}(\bullet,\bullet)$'s.  

In more detail we are here using   the general fact that if we have a MHS $M$ with graded pieces $A,B,C$ that are pure \hs s, then there is a fibration 
\[
\begin{matrix}
\text{Extension data for $M$}\\
\downarrow \\
\Ext^1_{\MHS}(B,A)\times \Ext^1_{\MHS}(C,B) \end{matrix}\]
whose fibres have connected components isomorphic to $\Ext_{\MHS}^1(C,A)$.   
To see this we have
\[
\xymatrix{0\ar[r]&B\ar[r]&M/A\ar[r]&C\ar[r]&0\\
0\ar[r]&A\ar[r]&\Ker(M\to C)\ar[r]&B\ar[r]&0}\]
and a master diagram
\[
\xymatrix@R=1.4pc@C=1pc{&0\ar[d]&0\ar[d]&&\\
&A\ar@{=}[r]\ar[d]&A\ar[d]&&\\
0\ar[r]& \Ker(M\to C)\ar[r]\ar[d]& M\ar[d]\ar[r]&C\ar[r]\ar@{=}[d]&0\\
0\ar[r]&B\ar[r]\ar[d]&M/A\ar[r]\ar[d]&C\ar[r]&0\\
&0&0.}\]
 From \cite{Car80} the extension classes are
 \[
 e_{BC}\in \frac{\Hom_\C(C,B)}{F^0 \Hom_\C(C,B)+\Hom_\Z(C,B)},
 \quad e_{AB}\in \frac{\Hom_\C(B,A)}{F^0\Hom_\C(B,A)+\Hom_\Z(B,A)}.\]
 We also have the extension classes
 \[
 e_1\in \frac{\Hom_\C(C,\Ker(M\to C))}{F^0+\Hom_\Z},\quad e_2\in \frac{\Hom_\C(M/A,A)}{F^0+\Hom_\Z}\]
 with
 \[
 e_1\to e_{BC},\quad e_2\to e_{AB}
 \]
 induced respectively by
 \[  \Hom(C,\Ker(M\to C))\to\Hom(C,B),\quad \Hom(M/A,A)\to \Hom(B,A).\]
 Knowing $e_{AB}$ and $e_1$ or $e_{BC}$ and $e_2$ determines $M$ as a MHS. For example,
 \beb
 \item $e_{AB}$ determines $\Ker(M\to C)$ and
 \item $e_1\in \{u\in \Hom(C,\Ker(M\to C))/F^0+H_\Z:u\to e_{BC}$ under $\ker(M\to A)\to B\}$.
 \eeb
 Thus if we know $e_{AB}$ and $e_1$, then we know the $u$ in the second bullet and from this we know~$M$.
 
 If $e_{BC}=0$, then the second bullet reduces to
 \[
 \frac{\Hom_\C(C,A)}{F^0\Hom_\C(C,A)+\Hom_\Z(C,A)}.\]
 In general we obtain a fibre space.
 
 We note that even though the action of monodromy on the associated graded of the MHS's is trivial, its action on extension data of various levels  is abelian and will generally not be trivial.

The key observation is the use of horizontality (transversality) to prove the 

\begin{Prop}\lab{2.7}
For $k\geqq 2$ the differentials of the maps $\Phi_{k,k+1}$ are zero.\end{Prop}

A proof of this result is given in \cite{GGR22} (cf.\ also \cite{Gr22}).  The following is a sketch of the argument.  There is a sequence of period mappings 
\begin{equation}
\lab{2.8}
\Phi_k:X\to\Ga_k\bsl D_k
\end{equation}
  to the extension data of levels $\leqq k$ in \eqref{2.5}.  Denoting by $X_k\subset X$ a  typical  fibre of $\Phi_{k-1}$, we have
 \[
 \Phi_{k-1,k}: X_k \to E_k \]
 where
 \[
 E_k\cong \opplus^\ell \Ext^1_{\MHS}(H^{k+\ell},H^\ell)\]
 with 
 \[
\Ext^1_{\MHS} (H^{k+\ell},H^\ell) =   \frac{\Hom_\C(H^{k+\ell},H^\ell)}{F^0\Hom_\C(H^{k+\ell},H^\ell)+\Hom_\Z(H^{k+\ell},H^\ell)}.\]
 
 By horizontality the differential of $\Phi_{k-1,k}$ gives a map
 \[
 T X_k\to H_k \subset TE_k\]
 where the following diagram depicts the complex tangent spaces $TE_k$ with the subspaces $H_k$ being the part over the red:
\begin{align*}
k=1&\quad \underbrace{(\ell-1,-\ell)\oplus\cdots\oplus (0,-1)}_{F^0} \oplus
\underbrace{ 
	 {\color{red}\underbrace{
	 	{\color{black} (-1,0)}}} \oplus\cdots\oplus (-\ell,\ell-1)}_{TE_1}\\
k=2&\quad \underbrace{(\ell-2,-\ell)\oplus\cdots\oplus (0,-2)}_{F^0} \oplus
\underbrace{ 
	 {\color{red}\underbrace{
	 	{\color{black}(-1,-1)}}} \oplus (-2,0)\oplus \cdots\oplus (-\ell,\ell-2)}_{TE_2}\\
k=3&\quad \underbrace{(\ell-3,-\ell) \oplus\cdots\oplus (0,-3)}_{F^0} \oplus 
\underbrace{ 
	 {\color{red}\underbrace{
	 	{\color{black}(-1,-2)}}}\oplus (-2,-1) \oplus\cdots\oplus (-\ell,\ell-3)}_{TE_3}.
		\end{align*}
Denoting any of the $E_k$ by $E$, as a mapping of real manifolds we have 
\[
T_\R E \cong \R^{2m}\]
and the differential of $\Phi_{k-1,k}$ maps to a subspace $H_\R \subset \R^{2m}$ that is invariant under the complex structure $J:\R^{2m}\to\R^{2m}$.  On the complexification of $\R^{2m}$ given by the action of $J$ the complexification of $H_\R$ is $H\oplus \ol H$ where $H$ is the term over the \raise5pt\hbox{${\color{red}\underbrace{\quad}}$} above.  Denoting by $\La$ the discrete subgroup of $H$ induced  by the $\Hom_\Z(H^{k+\ell},H^\ell)$ terms above we see that
\beb
\item $H/\La$ is a compact complex torus for $k= 1$;
\item $H/\La\cong (\C^\ast)^m$ for $k=2$;\footnote{In this case we have a fibration over a compact, complex torus whose connected components of the  fibres are isomorphic to $(\C^\ast)^m$'s.  This does \emph{not} mean that the total space is a semi-abelian variety.}
\item $H\cong \C^n$ for $k\geqq 3$.
\eeb

The maps (II.6) correspond to maps arising from a diagram
\[
\xymatrix{\pi_1(X_k)\ar[dr]\ar[rr]& &\La\\
&H_1(X_{k},\Z).\ar[ur]&}\]
Using a weight argument this implies that $TX_k$ maps to a sub-Hodge structure of $TE_k$ whose complex part lies under the red. It also means   that for a desingularization $X'_k$ of $X_k$ and smooth completion $\ol X'_k$ with $\ol X'_k\bsl X'_k=Z'_k$ a normal crossing divisor
\beb
\item the pullback by $\Phi_{k-1,k}$ of the dual $H^\ast $ of $H$ will give 1-forms in $H^0(\Om^1_{\ol X'_k}  (\log Z'_k))$, and
\item the mapping $\Phi_{k-1,k}$ will be given by integrating these logarithmic 1-forms taken modulo periods from $H_1(X'_k,\Z)$.
\eeb
For $k\geqq 3$ there are no periods and so $\Phi_{k-1,k}$ is locally constant. \hfill\qed

\begin{Cor}  \lab{2.9} In \eqref{2.5}  the induced map
\[
\Phi(X)\to\Phi_2(X)\]
is a finite morphism of algebraic varieties.\end{Cor}

\begin{proof}
It follows from \eqref{2.7} that the connected components of  fibres of the above map are points.  An analysis of how these maps are defined (they are iterated  integrals of logarithmic 1-forms taken modulo periods) shows that both $\Phi(X)$ and $\Phi_2(X)$ may be completed to algebraic varieties   and that $\Phi(X)\to\Phi_2(X)$ extends to a rational map on the completions (cf.\ \cite{GGR22} and \cite{Gr22}).
\end{proof}

\begin{proof}[Proof of Theorem \ref{2.4}] The first basic observation is that the tower of Albanese mappings \eqref{2.3} is \emph{not} the tower obtained by considering a higher Albanese mapping as a period mapping and then taking extension data of level $\leqq k$ for each $k$.  For example, the first level of the Albanese tower is
\[
\Alb(X)=H^0(X,\Om^1_X)^\ast/H_1(X,\Z)\]
and in general $H_1(X,\Z)$ is a \mhs\ with weights $-2,-1$ and for which there is non-trivial extension data (cf.\ \eqref{2.15} below).

A second point is that the period mappings \eqref{2.3} are only defined upon choice of a base point $x_0\in X$.  Thus the relevant Hodge structures are a $\Q$ in weight zero for the point and $H^1(X,\Q)$ for the Albanese variety.  The mapping $\alpha$ in \eqref{2.3} needs both of these (cf.\ the discussion in 5.32 ff.\ in \cite{SZ85}).

  We now change notation and denote by \eqref{2.8} the level of extension data mappings \hbox{associated} to UVMHS given by $\alpha_s:X\to\Alb^s(X)$ for $s\gg 0$.  Then $\Phi_2(X)$ contains the information in $\Alb_{X,x_0} :\allowbreak X\to\Alb(X)$. The $D_k$ in the tower of mappings \eqref{2.8} are obtained from
\begin{enumerate}[(i)]
\item the associated graded of the cohomology groups in $\wh{\pi_1(X,x)}$; and
\item the various levels of extension data among these groups, as explained above.
\end{enumerate}
Since we are dealing with a \emph{unipotent} VMHS, the terms in (i) are constant.  The basic observation from Corollary \ref{2.9} is that the only varying extension data in the period mapping \eqref{2.5} in this case is that arising from $\Alb_{X,x}$.  In terms of extension data arising from the UVMHS the information in $\Alb_{X,x}$ is contained in the mapping $\Phi_2$ to extension data of level $\leqq 2$. Thus 
\[
\quad\bmp{4.75}{\emph{If $Y\subset X$ is a compact subvariety such that $\alpha\big|_Y=\text{constant}$,\\ then $\alpha^s\big|_Y=\text{constant}$.}}\]
Since $\pi_1(X)$ is assumed to be residually nilpotent, this statement implies Theorem \ref{2.4} and~\eqref{1.7}.\end{proof}

\ssni{Proof of \eqref{2.2}}  Choosing a base point $x_0\in X$ and a basis $\om_\la$ for $H^0(\Om^1_{\ol X}(\log Z))$ the Albanese mapping of $X$ is given by
\begin{equation}\lab{2.10}
\alpha(x) = \left(\ldots, \int^x_{x_0}\om_\la,\ldots\right)\text{ modulo periods}.\end{equation}
The assumed completion of $X$ means that at every point $p$ of $Z$ there is an $\om_\la$ having a logarithmic  singularity at $p$.  In more detail, if in coordinates   $Z$ is locally given by
\[
z_i = 0,\qquad i\in I\]
then there is an $\om\in H^0(\Om^1_{\ol X}(\log Z))$ such that
\[
\om = \sum_{i\in I} a_i \frac{dz_i}{z_i} + \text{ holomorphic terms}\]
where all $a_i\ne 0$.
\medbreak

The pullback to $\wt X$ of $\Alb(X)$ is a trivial bundle $\wt X\times \C^N$.  Choosing coordinates $t_\la$ for $\C^N$ the pullback to $\wt X$ of $\om_\la$ is $dt_\la$.  Then by the above
\[
\vp = \sum |t_\la|^2\]
is an exhaustion function on $\wt X$; the projection to $X$ of the sets $\vp<C$ stay a fixed distance from $Z$.  Moreover, the Levi form $\lrp{\frac{i}{2}}\part\ol\part \vp$ is positive in the Zariski tangent spaces to the analytic subvariety $\wt \alpha (\wt X)\subset \C^N$.\hfill\qed

 \begin{subsec} \emph{An interpretation of the Albanese map as an extension class.}
   \lab{2.11} 
   
The following is intended to provide background for the discussion in the next section of the question raised in the introduction of detecting cohomologically the fibres of the \Shaf\ mapping in the VHS case.

We have noted that the analytic group
\[
\Alb(X)=H^0(\Om^1_{\ol X}(\log Z))^\ast/H_1(X,\Z)\]
is a semi-abelian variety that fits in an exact sequence
\begin{equation}\lab{2.12}
0\to T\to \Alb(X)\to \Alb(\ol X)\to 0\end{equation}
where $T\cong (\C^\ast)^m$ is an algebraic torus.  Completing $\Alb(X)$ by replacing the $\C^\ast$'s by $\P^1$'s the Albanese mapping on $X$ extends to a rational mapping; the graph of $\alpha$ in $X\times \Alb(X)$ closes up in $\ol X\times \ol{\Alb(X)}$. \end{subsec}

By way of an interlude we will first describe 
\begin{subsec}\lab{2.13}
\emph{The extension class of the semi-abelian variety $\Alb(X)$.} 

The exact sequence of connected abelian Lie groups \eqref{2.12} is constructed from the exact cohomology sequence
\begin{equation}\lab{2.14}
0\to H^1(\ol X,\Q) \to H^1(X,\Q)\to H^1(T,\Q)\to 0.\end{equation} 
The term in the middle has a \mhs\ with weight filtration $W_1 \subset W_2$ where
\beb
\item $\Gr^W_1 (H^1(X,\Q))=H^1(\ol X,\Q)$;
\item $\Gr^W_2 H^1(X,\Q)\cong H^1(T,\Q) \cong \opplus^m \Q(-1)$.\eeb
The extension class of \eqref{2.14} is in 
\begin{align*}
\Ext^1_{\MHS}(\opplus^m \Q(-1), H^1(\ol X,\Q))&\cong \opplus^m \lrc{
\frac{H^1(\ol X,\C)}{F^1 H^1(\ol X,\C)+H^1(\ol X,\Z)}}\\
&\cong \opplus^m H^1(\ol X,\cO_{\ol X}) / H^1(\ol X,\Z)\\
&\cong \opplus^m \Pic^\circ (\ol X).\end{align*}
From this we may infer that the extension class of \eqref{2.14}    gives  the image in $\Pic^\circ (\ol X)$ of the kernel of the Gysin mapping
\[
  \oplus^i H^0(\cO_{Z_i}) \xri{\mathrm{Gy}} H^2(X,\Z) .\]
\end{subsec}

Next we shall describe 

\begin{subsec}\lab{2.15} \emph{The extension class of $\Alb(X,x_0)$.}  

Thus far we have made  no reference to the choice of a base point $x_0\in X$.  When we include this as in \cite{Ha87b} we obtain a variation of \mhs\ with underlying local system  that we denote by $\V_{x_0}$.  When $X=\ol X$ is projective then for $x\in X$ the fibre $\V_{x_0,x}\cong H^1(X,\{x_0,x\};\Q)$  is a MHS with weight filtration $W_0\subset W_1$ where
\beb
\item $\Gr^W_0 \V_{x_0,x}=\Q$ (constant local system);
\item $\Gr^W_1 \V_{x_0,x}\cong H^1(X,\Q)$.\eeb
The extension class is in
\begin{align*}
\Ext^1_{\MHS}(H^1(X,\Q),\Q) &\cong \frac{H^1(X,\C)^\ast}{F^1 H^1(X,\C)^\ast+ H^1(X,\Z)^\ast}\\
&\cong H^0(\Om^1_X)^\ast/H_1(X,\Z).\end{align*}
It may be identified with the linear function on $H^0(X,\Om^1_X)$ given by the usual  Albanese mapping 
defined for $\om\in H^0(\Om^1_X)$ by 
\begin{equation}\lab{2.16}
\alpha_{x_0}(x)(\om) = \int^x_{x_0}\om\qquad \mod\text{periods}.\end{equation}
As in \cite{HZ87} this defines a canonical UVMHS associated to $(X,x_0)$.

In the quasi-projective case we again have a VMHS with local system $\V_{x_0}\to X$ where 
\begin{equation}\lab{2.17}
0\to \Q\to \V_{x_0}\to H^1(X,\Q) \to 0.\end{equation}
This time $\V_{x_0,x}\cong H^1(X,\{x_0,x\};\Q)$ is a MHS with weight filtration $W_0\subset W_1 \subset W_2$ with
\beb
\item $\Gr^W_0\V_{x_0,x}=\Q$;
\item $W_2 (\V_{x_0,x})/W_0(\V_{x_0,x})\cong H^1(X,\Q)$.\eeb
 \end{subsec}
\begin{Prop}\lab{2.18} The sequence \eqref{2.17} may be constructed by the same method as in \eqref{2.15} where the extension class is given for $\om \in H^0(\Om^1_{\ol X}(\log Z))$ by the Albanese mapping \eqref{2.16}.\end{Prop}

   What this means is that the prescription  for constructing the Hodge filtration $F^1_x$ on $\C\oplus H^1(X,\C)$ may be done using the same formulas as in the projective case above.   
It is when we put in the variable point $x\in X$ that we get a varying $F^p_x$ in the VMHS.

It is clear that the $F^1_x$ defines a holomorphic sub-bundle of $\V\otimes \cO_X$.  What must be verified is that horizontality for $F^1$ and conjugate horizontality for $\ol F^1$ are satisfied.  If $e_x\in \frac{H_1(X,\C)}{F^0H_1(X,\C)+H_1(X,\Z)}$ is the extension class, then the horizontality condition is $\nabla e \in F^{-1}H_1(X,\C)$ which  is satisfied in this case.  Conjugate horizontality is done similarly.\hfill\qed\medbreak

We have put this argument in because it will be used verbatim when we replace $\Q$ by $\V$ below.

\ssni{Third proof of I.7 and the $\part\ol\part$-lemma}

From rational homotopy theory it is known that the unipotent completion $\wh{\pi_1}(X)\otimes \Q$ is determined by $H_1(X,\Q)$, $H_2(X,\Q)$ and the map
\begin{equation}
\lab{2.19}
H_2(X,\Q)\to\La^2H_1(X,\Q)\end{equation}
that is dual to the cup product
\begin{equation}
\lab{2.20}
\La^2 H^1(X,\Q)\to H^2(X,\Q)\end{equation}
(cf.\ \cite{Mo78} and the references cited there).  Denoting by $\cL H_1(X,\Q)$ the free Lie algebra generated by $H_1(X,\Q)$, there is a surjection
\begin{equation}
\lab{2.21}
\cL H_1(X,\Q) \to \cL(\wh{\pi_1(X)}\otimes \Q) \end{equation}
where the right-hand side is the Lie algebra/$\Q$ corresponding to the completion of the group algebra of $\pi_1(X)$.  
We will use the construction of the Sullivan minimal model to determine the ideal $\cI$ in $\cL H_1(X,\Q)$ that gives the kernel of the mapping \eqref{2.21}. Both $\cL H_1(X)\otimes \Q$ and $\cI$ are   graded by degree and we denote by $(\bullet)_k$ the part up to degree $k$.  The key points will be that in building $\cI$ step-by-step with $\cI_2 \subset \cL H_1(X,\Q)_2$, $\cI_3\subset \cL H_1(X,\Q)_3,\dots$, at each step we will have a morphism of \mhs s
\[
\cL H_1(X,\Q)_k/\cI_k\to H^2(X,\Q),\]
and then a weight argument will give that $\cI$ is generated in degrees $2,3,4$ (\cite{Mo78}).  When $X=\ol X$ is projective the filtrations of $\cL H_1(X,\Q)$ by degree and weight conincide; this is not the case when $X$ is only quasi-projective.

For
$
\cG_\Q:=\cL H_1(X,\Q)/\cI$
we set $$\cG = \cG_\Q\otimes \C.$$  We next denote by $A^\bullet(\ol X,\log Z)$ the $C^\infty$ log complex generated by the smooth forms on $\ol X$ adjoined by the $dz_i/z_i$'s along $z_1\cdots z_k =0$.  From \cite{GS75} or \cite{Mo78} the inclusion $A^\bullet(\ol X,\log Z)\hookrightarrow A^\ast(X)$ is a quasi-isomorphism so that we have
\[
H^\ast (X;\C)\cong H^\ast (A^\bullet (\ol X,\log Z)).\]Assuming that $\pi_1(X)$ is residually nilpotent and with terms to be explained in the proof we have 
\begin{Prop}[\cite{Ha87b}]\lab{2.22}
There exists a form $\om\in \cG\otimes A^1(\ol X,\log Z)$ that satisfies
\begin{enumerate}[{\rm (i)}]
\item $d\om +\frac12[\om,\om]=0$;
\item integrating $\om$ defines the higher Albanese mapping;
\item the entries of $\om$ give the minimal model of $\wh{\pi_1(X)}\otimes \C$;
\item $\om_{-k}$ has $(\log z)^k$ singularities along $Z$;
and
\item the ideal $\cI\subset \cL H_1(X)$ is generated in degrees $(2,3,4)$.\end{enumerate}
\end{Prop}

\begin{Lem}[$\part\ol\part$ Lemma]\lab{2.23}
Let $\eta\in A^2(\ol X,\log Z)$ be $\part,\ol\part$ closed and with cohomology class $[\eta]=0$ in $H^2(X,\C)$.  Then there exists a function $f\in A^0(X)$ that satisfies
\[
\part\ol\part f=\eta.\]
We may choose this function to have  logarithmic singularities along $Z$.  Any such function is then unique up to $\Gr^W_2H^1(X,\C)$.\end{Lem}

The proof of the lemma will be given following that of Proposition \ref{2.22}.
Before giving that argument we will make some preliminary observations.  The first is that we shall use the standard identifications
\begin{equation}\lab{2.24}
\bcs {\rm (a)}& H^1(X,\C)\cong \bH^1 (\Obul_X(\log Z))\cong \overbrace{H^0(\Om^1_{\ol X}(\log Z))}^{F^1}\oplus H^1(\ol X,\cO_{\ol X})\\[6pt]
{\rm (b)}& H^2(X,\C)\cong \bH^2 (\Obul_{\ol X}(\log Z))\cong \overbrace{\overbrace{H^0(\Om^2_{\wt X}(\log Z))}^{F^2} \oplus H^1(\Om^1_{\ol X}(\log Z))}^{F^1} \oplus H^2 (\cO_{\ol X}).\ecs\end{equation}
In both cases
\[
H^q(\cO_{\ol X})\cong \ol{H^0(\Om^q_{\ol X})}.\]
A result of Deligne is that
\[
H^0(\Om^q_{\ol X}(\log Z))\text{ is a direct summand of }H^1(\ol X,\C);\]
i.e., forms in $H^0(\Om^q_{\ol X}(\log Z))$ are closed and never exact.  A proof of this using $A^\bullet(\ol X,\log Z)$  is in \cite{GS75}.

The Lie algebra $\cG$ is nilpotent and may be realized as a Lie algebra of lower triangular matrices.  We then have
\[
\om=\om_{-1}+\om_{-2} +\om_{-3}+\cdots\]
where $\om_{-k}$ are $1$-forms that on a diagonal  that is $k$-steps below the principal diagonal.  The equation (i) in Proposition \ref{2.22} is then a sequence of equations
\[
d\om_{-k}=-\lrp{\frac{1}{2}} \sum_{i+j=k} \om_{-i}\wedge\om_{-j},\quad k\geqq 0\; \footnotemark
\tag{II.25$_k$}
\]
\footnotetext{For $k=1$ the equation (II.25$_k$) is $d\om_{-1}=0$.}
We will argue that since
\begin{enumerate}[(i)]
\item $\wh{\pi_1}(X)\oplus \Q$ is determined by $H_2(X,\Q)\to\wedge^2 H_1(X,\Q)$;
\item the minimal model is isomorphic to the dual of $\wh{\pi_1}(X)\otimes \Q$; and
\item constructing the minimal model entails solving equations $\alpha=d\beta$\end{enumerate}
we will be able to inductively solve the equations (II.25$_k$).  Using the \mhs s on $H^1(X,\Q)$ and $\ker\{\wedge^2 H^1(X,\Q)\to H^2(X,\Q)\}$, from a weight argument we will see that the only ``new" equations  (II.25$_k$) arise for $k=2,3,4$ (\cite{Mo78}).

\begin{proof}[Proof of Proposition \ref{2.22}]
We take the entries of $\om_{-1}$ to be bases $\vp_i \in H^0(\Om^1_{\ol X}(\log Z))$, $\ol\psi_\alpha\in \ol{H^0(\Om^1_{\ol X})}$.  The (II.25$_k$) is
\[
d\om_{-2}=-\lrp{\frac12}\lrc{ \sum a_{ij}\vp_i\wedge\vp_j+\sum b_{i\alpha} \vp_i\wedge\ol\psi_\alpha+\sum c_{\alpha\beta}\ol\psi_\alpha\wedge \ol\psi_\beta}.\tag{II.25$_2$}\]
The entries of $\om_{-2}$ will correspond to a basis of the kernel of \eqref{2.20}$\otimes \C$.  From (II.24$_b$) the first and third entries in the right-hand side of (II.25$_2$) are zero as differential forms.  This leaves the middle term and by \eqref{2.23} we may take 
\[
\om_{-2}=-\ol\part f_{-2}\;\footnotemark\] \footnotetext{We could also take $\om_{-2}=\part f_{-2}$.  Then $[\om_{-1},\om_{-2}]$ is represented by closed forms with two different $(p,q)$ types.  Since it is supposed to have intrinsic Hodge-theoretic meaning this suggests (but does not prove) that we can solve the equations (II.25$_k$) for $k= 3$.}
and we have defined the part $(\cL H_1(X)_2/\cI_2)^\ast$ of the minimal model up through degree 2.

For the next step (II.25$_3$) is
\[
d\om_{-3} =-[\om_{-1},\om_{-2}].\]
We think of the right-hand side as giving a map
\[
(\cL H_1(X))_3  / \cI_2\cdot H_1(X)\to H^2(X).\]
Since $\pi_1\wh{(X)}\otimes\Q$ is accounted for by the kernel of \eqref{2.21},  this map must be zero  which means we can solve for $\om_{-3}$.  Using Lemma \ref{2.23} we may take 
\[
\om_{-3} = \ol\part f_{-3}.\]
This process continues inductively giving us (i) and (iii).

The proof of (ii) follows from the construction of the higher Albanese varieties in \cite{Ha87a}.

The proof of (iv) will follow from the proof of Lemma \ref{2.23}.

Finally, for (v) the maps
\[
\cL H_1(X)_k/\cI_k\to H^2(X)\]
that are used to inductively to  define  $\om_{-k}$ are morphisms of MHS's.  In (II.25$_k$) the weights of $\om_{-k}$ are the weights on the right-hand side.  In  (II.25$_k$) the correspondence  on weights between the left- and right-hand sides is
\begin{alignat*}{5}
k=2\qquad&&2,3,4&\longleftrightarrow 2,3,4&\\
k=3\qquad&&3,4,5,6&\longleftrightarrow 2,3,4&\\
k=4\qquad&&4,5,6,7,8&\longleftrightarrow 2,3,4&\\
k=5\qquad&&5,6,7,8,9,10&\longleftrightarrow 2,3,4.\end{alignat*}
It follows that these maps are zero for $k\geqq 5$.  Thus for $k\geqq 5$
\[
\cI_k=\cI_5 \{\ottimes^{k-5}H_1(X)\}.\qedhere\]
\end{proof}

\begin{proof}[Proof of Lemma \ref{2.23}]
We will use  \eqref{2.24}(b).  
Writing $\eta=\eta^{2,0}+\eta^{1,1}+\eta^{0,2}$ we have $\ol \part \eta^{2,0}=0$.  Thus
\vspace*{-3pt}\[
\eta^{2,0}\in H^0\lrp{\Om^2_{\ol X}(\log Z)}.\vspace*{-3pt}\]
Since the cohomology class $[\eta]=0$ in $H^2(X,\C)$,   it follows from the  result of Deligne that the differential form $\eta^{2,0}=0$. Similarly, $\eta^{0,2}=0$ and
\vspace*{-3pt}\[
\eta=\eta^{1,1}\in A^{1,1}(\ol X,\log Z)\vspace*{-3pt}\]
is $\part,\ol\part$ closed and with cohomology class
\vspace*{-3pt}\[
[\eta]=0\hensp{in}H^1\lrp{\Om^1_{\ol X}(\log Z)}.\vspace*{-3pt}\]

From the exact sequence 
\[
\oplus H^0(Z_i,\C)(-1)\xri{\rm Gy} H^2(\ol X,\C)\to H^2(X,\C)\]
since $[\eta]$ maps to zero in $H^2(X,\C)$, it follows that in $H^2(\ol X,\C)$
\[
[\eta]=\sum c_i\Gy(1_{Z_i}),\quad c_i\in \C.\]
It is well known, and a proof will be recalled below, that there is $g_i \in A^0(X)$ such that $\part\ol\part g_i$ gives a de~Rham representative of $\Gy(1_{Z_i})$.
  Then
\[
\alpha:= \eta-\sum c_i \part\ol\part g_i\hensp{is a} \part,\ol\part \hensp{closed (1,1) form whose  cohomology class $[\alpha]=0$ in} H^2(\ol X,\C).\]
A standard K\"ahler fact is that then $\alpha=\part\ol\part h$ for some $h\in A^0(\ol X)$.  This follows from the result  that the inclusion
\setcounter{equation}{25}
\begin{equation}\lab{2.26}
A^\bullet (\ol X)\cap \ker\part \hookrightarrow A^\bullet (\ol X)\end{equation}
is a quasi-isomorphism (Section 5 in \cite{DGMS75}). 

For the statement about $\Gy(1_{Z_i})$, in the line bundle $L_i:=[Z_i]$ we choose a metric and a section $\sigma_i\in H^0(\ol X,L_i)$ with divisor $(\sigma_i)=Z_i$.  If $\rho_i$ is a $C^\infty$ bump function with $\rho_i\big|_{Z_i}=1$ and that is compactly supported in a neighborhood of $Z_i$, then
\[
\lrp{\frac{\sqrt{-1}}{2\pi}} \part\ol\part \log \|\rho_i\sigma_i\|^2\]
represents $\Gy(1_{Z_i})$ in $H^2(X,\C)$. 
\end{proof}

The statement about the singularities of $f$ is then clear.

If we have $f$ with $\part\ol\part f=0$, then $\part f\in A^1(\ol X,\log Z)$ gives a cohomology class in $H^1(X)$ such that $\Res[\part f]\in\ker\{\opplus^i H^0(C_{Z_i}) (-1)\to H^2(\ol X)\}\cong \Gr^W_2 H^1(X,\C)$, which implies the uniqueness statement.

Finally to prove Theorem A as in the first argument we may reduce to showing that for a nodal curve $Y=\cup Y_\alpha$ 
\[
H^1(X,\Q)\to H^1(Y,\Q)\hensp{is trivial} \implies \wh{\pi_1}(Y)\to \wh{\pi_1}(X) \hensp{is trivial.}\]
The assumption gives that all
\[
\om_{-1}\big|_{Y_\alpha}=0.\]
From (II.23$_k$) it follows that
\[
d\om_{-2}\big|_{Y_\alpha}=0.\]
From  Lemma \ref{2.23} we have $\om_{-2}=df_{-2}$ where $\ol\part f_{-2}=0$.  Then
\[
\part\ol\part f_{-2}\big|_{Y_\alpha}=0\implies f_{-2}\big|_{Y_\alpha}=\text{constant}\implies \om_{-2}\big|_{Y_\alpha}=0.\]
Proceeding inductively on $k$ we obtain
\[
\om\big|_{Y_\alpha}=0,\]
which implies the result that all maps $\wh{\pi_1}(Y_2)\to \wh{\pi_1}(X)$ are trivial.  From \cite{Ha87b} the higher Albanese mappings
\[
X\to \Alb^s(X)\]
contract all $Y_\alpha$  and hence send $Y$ to a point.\hfill\qed

\begin{rem}
If $\wt \om$ is the pullback of $\om$ to the universal cover $\wt X\xri{\pi} X$, then it is well known that there is a
linear complex Lie group $G$ with Lie algebra $\cG$ and a mapping
\begin{equation}\lab{2.29}
\wt f :\wt X\to G\end{equation}
with
\[
\wt f^\ast (g^{-1}dg)=\wt \om.\]
If $\Ga\subset G$ is the monodromy group of the flat connection $\om$, then \eqref{2.29} induces a map
\[
f:X\to\Ga\bsl G.\]
From \cite{Ha87b} there is a subgroup $F^0G$ such that the induced  map
\[
 X\to \Ga\bsl G/F^0G\]
is the Albanese map for $s\gg 0$.  This gives another argument that
\[
\om\big|_{Y_\alpha}=0\implies \alpha_s(Y)=\text{point}.\]
\end{rem}

Finally one may ask about the singularities of the locally defined matrix-valued function $g$ along $Z$.  Consider the first non-trivial case where locally on $X$
\begin{align*}
g&= \bpm 1&0&0\\ a_1&1&0\\ b&a_2&1\epm\\
\om&= g^{-1}dg=\om_{-1}\oplus \om_{-2}.\end{align*}
From our construction, $\om_{-1}$ has $dz_i/z_i$  and $d\bar z_i/\bar z_i$ terms so that $a_1,a_2$ can have linear terms in  $\log z_i,\ol{\log z_j}$.  From
\[
\om_{-2}=db-a_2da_1\]
and the above construction of $\om_{-2}$, it follows that $b$ can have quadratic terms in $\log z_i,\log \bar z_j$.  In general
\[
\text{$\om_{-k}$ \emph{is a polynomial of degree $k$ in the $\log z_i$'s and their conjugates.}}\]
This illustrates a general result (\cite{Gr22}) about the period matrices of the extension data in a VMHS.

\section{The variation of Hodge structure and \mhs\ cases} \setcounter{equation}{0}

\begin{proof}[Proof of Theorem B] The proof is basically an observation using known results.  

The assumption that the monodromy representation
\[
\rho:\pi_1(X)\to\Aut(V)\]
has a  finite kernel is used in two ways.  One is that the covering of $X$ corresponding to $\pi_1(X)/\ker\rho$ is a finite quotient of the universal covering.  The other is that for a connected compact subvariety $Y\subset X$
\begin{equation}\lab{3.1}
\pi_1(Y)\to \pi_1(X)\text{ is finite }\iff \rho\big|_{\pi_1(Y) }\text{ is finite.}\end{equation}
It is a classical result from Hodge theory that
\[
\rho\big|_{\pi_1(Y)} \text{ is finite } \iff \Phi\big|_Y \text{ is constant.}\]
 Thus in \eqref{1.9} the connected fibres of $\wt \Phi:\wt X\to\wt P$ are the connected components of $\pi^{-1}(Y)$ where \eqref{3.1} is satisfied.

It remains to prove that in \eqref{1.9} the image $\wt \Phi(\wt P)\subset D$ is Stein.  We denote by $\check D$ the compact dual of $D$ (\cite{CM-SP17} and \cite{GGK13}).  Then
\begin{align*}
D&= G_\R/H\\
\cap&\\
\check D&= G_\C/B=M/H\end{align*}
where $G_\R$ is a real semi-simple Lie group, $G_\C$ is its complexification, $H$ is a compact subgroup of $G_\R$ and $M$ is the maximal compact subgroup of $G_\C$.  On $D$ there is a $G_\R$-invariant volume form $\Om_D$, and on $\CD$ there is an  $M$-invariant volume form $\Om_{\CD}$.  From \cite[Lecture 6]{GGK13}  the ratio
\[
\vp=\Om_{D}/\Om_{\CD}\]
is then a function on $D$ with the property   
\begin{quote}
$\vp:D\to\R$ is an exhaustion function and the Levi form $(i/2)\part\ol\part \vp$ is positive in the horizontal sub-bundle $I\subset TD$.
\end{quote}
From the assumed completeness of the period mapping $\Phi$ on $X$ it follows that the image $\wt \Phi(\wt X)=\wt P\subset D$ is a closed analytic subvariety of $D$  whose Zariski tangent spaces lie in $I$.  Consequently
\[
\vp\big|_{\wt P}\text{ is a psh exhaustion function}\]
and then by \cite{Na62} we may conclude that $\wt P$ is a Stein variety.\end{proof}

\begin{proof}[Proof of Theorem C]
We have outlined the proof in Section I; here we will give some further details.  To give a rigorous argument involves technical issues in the theory of complex analytic geometry that we shall not deal with here.  Ones that are similar arise in the projective $X=\ol X$ case and are treated in Section 5 in \cite{EKPR12}.

By the argument just given the image $\wt \Phi(\wt X)=\wt P$ is a Stein variety.  This uses the assumption that $\Phi_M$ is proper, which implies that the monodromies of $\Phi$ around the $Z_i$ are of infinite order.

Using the terminology and notations from Section II, we denote by $\Phi_{M,2}$ the mapping given by taking the associated graded together with the  extension data  of levels $\leqq 2$ in the MHS's described by \eqref{1.10}.  The arguments given in the proof of  \eqref{2.4} apply to the VMHS along the fibres of  $\Phi_M(X)\to \Phi(X)$.  Thus the map
\[
\Phi_M(X)\to \Phi_{M,2}(X)\]
has finite fibres.  Moreover, the fibres of
\[
\Phi_{M,2}(X)\to \Phi(X)\]
are images of the mappings of the fibres to extension data of levels $\leqq 2$.  Due to the assumed completeness of $\Phi_M$ their universal coverings are Stein. From this we see that
\begin{equation}\lab{3.2}
\wt \Phi_M(\wt X)\to \wt \Phi(\wt X) \end{equation}
is holomorphic mapping onto a Stein variety and with Stein fibres.   As in Section 5 of \cite{EKPR12} one may use \cite{Car79} and \cite{Se53} to conclude that the total space is Stein.
 \end{proof}

 \ssni{Comments regarding possible cohomological descriptions of the fibres of $\Phi$; deformations of the monodromy representation $\rho$}  
 
 The connected components $Y\subset X$ of the fibres of a \Shaf\ map are defined homotopy-theoretically by the finiteness of the map $\pi_1(Y)\to \pi_1(X)$.  One may ask if they can be characterized cohomologically as is possible in the nilpotent case by (iv) in Theorem A.  This is an interesting question that we shall now discuss in the case of a VHS.
 
 In order to be able to cohomologically  describe the fibres some cohomology groups must be non-zero, and we note that
 \begin{demo}\lab{new3.3}
 (i) \emph{Given the local system $\V\to X$ underlying a VHS, the group $H^1(X,\End(\V))$ is zero if, and only if, the monodromy representation
 \[
 \rho:\pi_1(X)\to\Aut(V_\C) \]
 is to first order rigid.}
 
 (ii) \emph{The map $H^1(X,\End(\V))\to H^1(Y,\End(\V))$ is  trivial if, and only if, the corresponding first order deformation of $\rho$ induces a trivial deformation of $\rho\big|_{\pi_1(Y)}$.}\end{demo}
 
 As a first step we will reformulate the cohomological criterion in the nilpotent case.  Keeping our assumption that $X=\ol X\bsl Z$ we denote by $\cO( \widehat{\pi_1(X,x))}:= \lim_s \cO(\pi_1(X,x)/J^{s+1})$ the Hopf algebra of functions on the unipotent completion of $\pi_1(X,x)$.  This algebra has a \mhs\ that is generated by $H^1(X,\Q)$ viewed as linear functions on $\pi_1(X,x)$ (\cite{Ha87b}).  For  $Y$ a normal crossing variety the induced mapping
 \[
 \widehat{\pi_1(Y,y)} \to \widehat{\pi_1(X,x)},\qquad f(y)=x\]
 of a map $f:Y\to X$  
 is trivial if, and only if, the pullback map
 \[
 f^\ast:\widehat{\cO(\pi_1(X,x))}\to \widehat{\cO(\pi_1(Y,y))} \]
 is zero.  This happens when this map is trival on generators; i.e., when
 \[
 H^1(X,\Q)\xri{f^\ast} H^1(Y,\Q)\]
 is zero.  We have used this in the proof of Theorem A.
 
 In the VHS case we let $G\subset \Aut(V)$ be the $\Q$-Zariski closure of the image of the monodromy representation.  Then $G$ is a semi-simple $\Q$-algebraic group.  For the purposes of this discussion we shall assume that it is simple and that the members of the set $\{V_\la\}$ of irreducible $G$-modules all arise as   sub-quotients of tensor products of $V$.  Then each $V_\la$ gives a local system $\V_\la\to X$ that underlies a VHS.
 
One now uses the completion of ${\pi_1(X,x)}$ relative to $\rho$ as in \cite{Ha98}.  Following the notations there, except that here we use $G$ instead of $S$, the relative completion is a pro-algebraic group $\cG$ that fits in a diagram
 \begin{equation}\lab{3.3}
 \bsp{
 \xymatrix{
 	1\ar[r]&\cU\ar[r]& \cG\ar[r]&G\ar[r]&1\\
	&&\pi_1(X,x)\ar[u]_{\wt \rho}\ar[ur]^\rho}}\end{equation}
where $\cU$ is the unipotent radical of $\mathcal{G}$.  The Hopf algebra $\cO(\mathcal{G})$ of regular functions on $\mathcal{G}$ has a functorial \mhs\ with non-negative weights and relative to which the Hopf algebra operations are morphisms.

The regular functions $\cO(G)$ may be realized as the matrix coefficients of the above representations.\footnote{From here on in this section we will work over $\C$ and will set $G_\C=G$, $\V_\la,\C =\V_\la$ etc.}
  We then have
\[
\cO(G)\cong \opplus^\la V^\ast_\la \otimes V_\la\]
  where $G$ acts on both the left and right.  Referring to \eqref{3.3} we have an inclusion $\cO(G)\hookrightarrow \cO(\cG)$ and
\begin{equation}\lab{3.4}
\cO(G) =\Gr^W_0 (\cO(\cG)).\end{equation}
From \cite{Ha98} it may be inferred that
\begin{equation}\lab{3.5}
\Gr^W_1 (\cO(\cG)) \cong \opplus^\la H^1(X,\V^\ast_\la)\otimes V_\la.\end{equation}
The reason that  $\V^\ast_\la$ is  inside the parenthesis and $V_\la$ is outside reflects the left and right actions of $G$ on $\cO(G)$ (cf.\ loc.\ cit.).

Given an inclusion $j:Y\hookrightarrow X$ from \eqref{3.3} and \eqref{3.4} the condition that
\begin{equation}\lab{3.6}
\wt \rho\circ j:\pi_1(Y)\to\cG\end{equation}
be trivial is
\[
(\wt \rho\circ j)^\ast \cO(\cG)=\text{ constant functions.}\]
From \eqref{3.5} this implies that all restriction mappings
\begin{CEquation}\lab{3.7}
H^1(X,\V_\la)\to H^1(Y,\V_\la)\end{CEquation}
are zero.  

\begin{Quest} \lab{3.8} If \eqref{3.7} holds for all non-trivial $V_\la$, then does \eqref{3.6} hold?
\end{Quest}

A positive answer would necessitate that
\begin{equation}
\lab{3.9} H^1(X,\V) \ne 0\end{equation}
for some $V=V_\la$.  This group has a MHS with  the part not in $H^1(\ol X,j_\ast \V)$ having a description involving the part of $\V$ that is not invariant under monodromy around   the $Z_i$'s (cf.\ \cite{Ha98} and \cite{Le21}).  We will not get into this here.  \emph{For the remainder of Section III we will assume that $X=\ol X$ is projective.}  Then $H^1(\ol X,\V)$ has a Hodge structure with 
\begin{equation}\lab{3.10}
F^1 H^1(X,\V)=H^0(\Om^1_{\ol X}).\end{equation}
   The non-zero vector space \eqref{3.10} gives rise to an interesting construction that we now shall discuss  (cf.\ loc.\ cit.\ for a related construction of locally constant iterated integrals).

\begin{subsec}\lab{3.11}
\emph{Twisted Albanese map.} 

We first note that for any local system $\V\to X$ and $\om\in H^0(\Om^1_{\ol X} \otimes \V)$ the exterior derivative $d\om$ is well defined.  If $\V\to   X$ underlies a VHS, then it may be shown that $d\om =0$.
\end{subsec}

Now let $x_0\in X$ and $\wt x_0\in\wt X$ be base points with $(\wt X,\wt x_0)\xri{\pi}(X,x_0)$.  Then canonically
\[
\pi^\ast \V \cong \wt X\times V.\]
If $\om$ is as  above with $\pi^\ast \om = \wt \om$, then $\wt \om \in H^0(\Om^1_{\wt X}\otimes \V)\cong H^0(\Om^1_{\wt X})\otimes  V$ and since $d\wt \om = 0$,
\begin{equation}\lab{3.12}
\int^{\wt x}_{\wt x_0} \wt \om \in V\end{equation}
is well defined.
Taking into account the action of $\pi_1(X,x)$ as deck transformations we may descend \eqref{3.12} to $X$ to have a well-defined mapping 
\begin{equation}\lab{3.13}
\Alb_{X,\V}: X\to H^0(\Om^1_X  \V)^\ast\otimes V_{x_0}/\wt\rho(\pi_1(X,x_0)).\end{equation}
This is the \emph{twisted Albanese mapping} referred to above.

We note that
\begin{equation}\lab{3.14}
\text{\emph{\eqref{3.7} is equivalent to $\Alb_{X,\V_\la}\big|_Y=$ constant.}}\end{equation}
Consequently an affirmative answer to \eqref{3.8} would mean that we have
\[
\rho\big|_{\pi_1(Y)} \text{ is finite $\iff \Alb_{X,\V_\la}\big|_Y=$ constant for all $V_\la$.}\]

\demo{\em Extensions of \mhs s and the first step in the construction of the \mhs\ on a relative completion.}\lab{3.15}

This   is an extension to the relative completion of a representation of  the fundamental group of the discussion in \eqref{2.13}, \eqref{2.15} and of Proposition \ref{2.18}.

\begin{Prop}\lab{3.16}
Let $\V\to X$ be a local system underlying a VHS where $\V_\la$ has weight zero.  Choosing a base point $x_0\in X$ there is for each $x\in X$ a canonical \mhs\ $\V_{1,x}$ with
\beb
\item $\Gr^W_0(\V_{1,x})=\C$;
\item $\V_{1,x}/W_0(\V_{1,x})\cong V^\ast_{x_0}\otimes  H^1(X,\V)$
\eeb
and that is constructed using $\Alb_{X,\V}(x)$ as extension class.  Letting $x$ vary we obtain an admissible VMHS.
     \end{Prop}

\begin{proof}
As was the case for Proposition \ref{2.18}, the construction of a \mhs\ from a class in $H^0(\Om^1_{\ol X}  \V)^\ast \otimes V_{x_0}/\wt\rho(\pi_1(X,x_0))$ may be done following the prescription in \cite{Car80}.  Using the same argument as in the $\V=\C$  we obtain a local system (flat vector bundle) and horzontality (transversality) of the family of \mhs s parametrized by $x\in X$ as consequences of
\[
H^1(X,\V_\C)\cong H^0(\Om^1_{\ol X} \otimes \V)\oplus \ol{H^0(\Om^1_{\ol X}\otimes \V)}\]
where the differential forms appearing in both summands on the right-hand side are closed.

Alternatively we may proceed as follows.  In general suppose we have over $X$
\[
0\to A\to B\to C\to 0\]
where $A,C$ are Hodge structures and $B$ is a family of \mhs s $B_x$ defined by
\[
e_x\in \frac{\Hom_\C(C,A)}{F^0\Hom_C(C,A)+\Hom_\Z(C,A)}.\]
The connection on $B$ is
\[
\nabla_B=\bpm
\nabla_A&\om\\
0&\nabla_C\epm\]
where
\[
\om\in \Hom_\C(C,A)\otimes \Om^1_X.\]
From 
\[
\nabla^2_B=\bpm
\nabla_A & \nabla_A\om+\om\nabla_C\\
0&\nabla^2_C\epm
\]
to have integrability we need $\nabla_A \om +\om \nabla_B =0$, which is
\[
\nabla_{\Hom(C,A)} \om=0.\tag{i}\]
For horizontality a calculation shows that we need
\[
\nabla_{\Hom(C,A)} e \in \frac{F^{-1}\Hom(C,A)\otimes \Om^1_X}{\nabla F^0\Hom(C,A)}.\tag{ii}\]
This is well defined since $\nabla F^0\subseteq F^{-1}$.  Taking 
\[
\om = \frac{\nabla_{\Hom(C,A)} e}{\nabla_{\Hom(C,A)}F^0\Hom(C,A)}\]
we see that both (i) and (ii) are satisfied.\end{proof}

For $\V\to X$ as in Proposition \ref{3.16}
 we have the associated local system $\End(\V)$.  We note that here there are two interpretations of $H^1(X,\End(\V))$.  One is as was just discussed in classifying extensions
 \begin{subequations}
 \begin{equation}\lab{3.18a} 0\to \End(\V)\to \M\to \End(\V)\to 0.\end{equation}    
 The other is as first order deformations of $\V\to X$.  One may think of the latter as local systems over a scheme $(X,\cO_{X,1})$ where $\cO_{X,1}$ is an extension of $\cO_X$ by nilpotents of order~2.  In the case at hand we may take the universal first order deformations to be the middle term in an exact sequence
 \begin{equation}
 \lab{3.18b}
 0\to \V\otimes H^1(X,\End(\V))^\ast \to \V_1\to \V\to 0.\end{equation}
 \end{subequations}
 The extension class defining this sequence is the tautological class given by the identity in
 \[
 H^1(X,\Hom(\V,\V\otimes H^1(X,\End(\V))^\ast)\cong H^1(X,\End(\V))\otimes H^1(X,\End(\V))^\ast\]
 and $\cO_{X,1}=\cO_X[H^1(X,\End(\V))^\ast]$.
 
 \begin{Prop}\lab{3.19}
 There is a VMHS over the scheme $(X,\cO_{X,1})$ whose restriction to $(X,\cO_X)$ is the VMHS defined by the class
 \[
 \Alb_{X,\End(\V)}\subset \Ext^1_{\MHS}(\V,\V\otimes H^1(X,\End(\V))^\ast).\]
 \end{Prop}
 
 \begin{proof}
 The proof is the same as that for Proposition \ref{3.16} above.  We note that the connection form for $\V_1$ as a local system over $\cO_X$ is
 \[
 \nabla_{\V} = \bpm
 \nabla_{\V\otimes H^1(\End(\V))^\ast}&\Om\\
 0&\nabla_{\V}\epm\]
 where $\Om\in H^0(X,\Om^1_X \otimes \Hom(\V,\V\otimes H^1( \End(\V))^\ast)$.  Then
 \[
 \nabla^2_V=0\iff \nabla_{\Hom}\Om +\frac12 [\Om,\Om]=0\]
 where the $\Hom$ is $\Hom(\V,\V\otimes  H^1(\End(\V))^\ast)$. 
 
 As a local system over $\cO_{X,1}$ where $H^1(X,\End(\V))^\ast$ is treated as deformation parameters since $[\Om,\Om]$ is quadratic  we have $\nabla^2_{\V_1}=0$. \end{proof}
 
 For a subvariety $Y\subset X$ the mapping
 \[
 H^1(X,\End(\V))\to H^1(Y,\End(\V))\]
 is interpreted as restricting first order deformations of $\rho:\pi_1(X)\to\Aut(V)$ to $\pi_1(Y)$. With this interpretation we have completed the proof of Proposition \ref{new3.3} and its corollary.

 The above may be summarized as follows:
 \begin{demo}\lab{3.21}
 \emph{Given a local system $\V\to X$ supporting a VHS, denote by $\cO_{X,1}$ the scheme $\cO_X[H^1(X,\End(E))^\ast]$ obtained by adjoining to $\cO_X$ nilpotents of order 2.  Then
 \begin{enumerate}[{\rm (i)}]
 \item $\cO_{X,1}$ has a \mhs\ with
 \[
 \bcs
 \Gr^W_0 (\cO_{X,1})=\C,\\
 \Gr^W_1(\cO_{X,1})=H^1(X,\End(\V))^\ast\ecs\]
 and whose extension class is $\Alb_{X,\End(\V)}$, and 
 \item over $(X,\cO_{X_1})$ there is a VMHS as described in Proposition \ref{3.16}.\end{enumerate}}\end{demo}
 
 In the projective case the extension of the above to all orders of deformation of $\rho$ (Goldman-Millson theory) is carried out in \cite{EKPR12}.  Thus there are
 \beb
 \item a MHS on the completed local rings $\wh \cO_{\Def(\rho)}$ of the Kuranishi space $\Def(\rho)$; and
 \item a formal VMHS over $\Def(\rho)$.
 \eeb
  For further discussion of these results  cf.\ \cite{Pr17}, \cite{Pr19}, \cite{Le19} and \cite{Le21}.
 
 \section{Future directions}
 \setcounter{equation}{0}

In this section we will very briefly outline some future research directions.
We have  made an attempt  to  formulate a categorical version of 
Shafarevich conjecture in \cite{HKL17},  \cite{KL14}.
In those papers we outline new techniques of perverse sheaves of categories and functors between them. The categorical
information is recorded by the skeleton and the 
Lagrangian sheaves over it. We  give an interpretation 
of the infinite chain conjecture.
 
\begin{center}
    \begin{tikzpicture}
        \begin{scope}[shift={(0,3)}]
            \draw(0,-1) node{\underline{Allowed}}
            (-3,-0.6) node{$C$}(-3,0.6) node{1-core};

            \draw[-stealth] (-1,0) to (1,0);
            \draw (-4,0)--(-3,0)--(-2,0.5)(-3,0)--(-2,-0.5);
           \begin{scope}[shift={(3,0)}]

              \newdimen\R
        \R=1cm
            \draw (0:\R)
             \foreach \x in {60,120,...,360} {  -- (\x:\R) }
                 -- cycle (360:\R)
             -- cycle (300:\R) 
              -- cycle (240:\R) 
             -- cycle (180:\R) 
                -- cycle  (120:\R) 
                -- cycle  (60:\R) ;
                \draw (0,-1.2cm) node{2-core};
                \draw (120:\R) --(240:\R) (180:\R)--(-0.5,0)
                (-0.75,-0.45cm)--(0,-0.45cm) node[scale=0.8]{$\bullet$}--(0,-0.88cm)
                (0,-0.45cm)--(30:0.88cm)
                
                ;

           \end{scope}
        \end{scope}
        \begin{scope}[shift={(0,0)}]
            \draw(0,-1) node{\underline{Not Allowed}};
            \draw[-stealth] (-1,0) to (1,0);
            \draw (-4,0)--(-3,0)--(-2,0.5)(-3,0)--(-2,-0.5);
           \begin{scope}[shift={(3,0)}]

              \newdimen\R
        \R=1cm
            \draw (0:\R)
             \foreach \x in {60,120,...,360} {  -- (\x:\R) }
                 -- cycle (360:\R)
             -- cycle (300:\R) 
              -- cycle (240:\R) 
             -- cycle (180:\R) 
                -- cycle  (120:\R) 
                -- cycle  (60:\R) ;
                \draw (-0.2,-0.2) node[scale=0.8]{$\bullet$}--(0.2,-0.2) node[scale=0.8]{$\bullet$};
                \draw (-0.2,0.2) node[scale=0.8]{$\bullet$};
           \end{scope}
        \end{scope}
    \end{tikzpicture}
\end{center}
\begin{Conj}[Categorical Shafarevich]
    Let $X$ be a quasi-projective variety with infinite (nilpotent) fundamental group. Assume that for any curve $C$ the functor
    \[F_{\text{wrapped}} (C) \to F_{\text{wrapped}} (X)\]
    is allowed. Then $X$ is holomorphically convex.
\end{Conj}

Here allowed maps  are maps those which map the core of the category of an open 
Riemann surface in the connected 1 dimensional subcore of the quasipojective variety. Details can be found in
the above references.

The above conjecture  suggests

\begin{Conj}
    Let $X$ be quasi-projective. Assume that every proper morphism $f$ from a smooth curve $C$ to $X$ has a meridian of infinity of infinite order in $\pi_1(X)$. Then is $X$  holomorphically convex?
\end{Conj}

Successful steps in this direction were taken by P. Eyssidieux and R. Aguilar. 

Other  considerations naturally lead to.

\begin{Quest}
Assume that $X$ is Picard hyperbolic and $\pi_1(X)$ is nilpotent. 
Does this imply  that $X$ is holomorphically convex?
\end{Quest}

More plausible may be a positive answer to  the following:

\begin{Quest}
Assume that $X$ is Picard hyperbolic and $\pi_1(X )$ is a big fundamental group in the sense of \cite{Br22}. Does this imply that
$X$ holomorphically convex (Stein).
\end{Quest}

  This question extends the work of  loc.\ cit.\ who considered 
  the case of a projective $X=\ol X$. In 
 fact it will be interesting to investigate how the following three notions are related.

\begin{center}
    \begin{tikzpicture}
        \node [draw,rounded corners=0.2cm,clip] (a) at (0,0) {$X$ is
hyperbolic};
        \node [draw,rounded corners=0.2cm,clip] (b) at (6,0) {$X$ is 1-formal};
        \node [draw,rounded corners=0.2cm,clip] (c) at (3,-2) {$\widetilde{X}$ is holomorphically convex};
        \path[stealth-stealth] (a) edge node[above]{?} (b)   (a) edge node[above]{?} (c) (b) edge node[above]{?} (c);
    \end{tikzpicture}
\end{center}

Recently several outstanding results were obtained in direction
of algebraicity of the period map in the case of variations
of mixed Hodge structures;  see \cite{BBT06}. 
It seems likely that one can drop
 the condition about infinite monodromy around $2$. We would like to pose the following.

\begin{Quest}
    Let $X$ be a quasi-projective variety such that  there is an inclusion
    \[\pi_1(X)  \hookrightarrow{} \GL(n,\C). \]
    Then is $\widetilde{X}$  holomorphically convex?

\end{Quest}

Several results in this direction were obtained by B. Brunebarbe and Y. 
Deng. Similarly 
in the case when the
fundamental
group $\pi_1(X)$
is subgroup of a nilpotent group of $\Nil p$ we have  the following:

\begin{Quest}
     Let $X$ be a quasi-projective variety such that 
    \[\pi_1(X)  \hookrightarrow{} \Nil p. \]
    Then is $\widetilde{X}$  holomorphically convex?
\end{Quest}

\end{document}

%% file: pgmacs.tex
 \usepackage{graphicx}
 \usepackage{epstopdf}
     \usepackage{amssymb}
  \usepackage{color} 

 \newcommand\MHS{{\mathrm{MHS}}}

\long\def\symbolfootnote[#1]#2{\begingroup%
\def\thefootnote{\fnsymbol{footnote}}\footnote[#1]{#2}\endgroup}

\newcommand\CD{\check{D}}

    \newcommand\Def{\mathop{\rm Def}\nolimits}

\newcommand\Obul{\Omega^\bullet}

\newcommand\bsm{ \begin{smallmatrix}}
\newcommand\bspm{ \left(\begin{smallmatrix}}
\newcommand\esm{\end{smallmatrix} }
\newcommand\espm{\end{smallmatrix} \right)}
\newcommand\bbm{\left[\begin{matrix}}
\newcommand\ebm{\end{matrix}\right]}
\newcommand\bcs{\begin{cases}}
\newcommand\ecs{\end{cases}}

\newcommand{\lrc}[1]{\left\{ #1\right\}}

\newcommand{\lrp}[1]{\left(#1\right)}

 \newcommand\wt[1]{\widetilde{#1}}
  \newcommand\wh[1]{\widehat{#1}}

 \newcommand\V{\mathbb{V}}

\newcommand{\C}{{\mathbb{C}}}

\newcommand\bH{\mathbb{H}}

\newcommand{\M}{\mathbb{M}}
 
\renewcommand{\P}{\mathbb{P}}
\newcommand{\Q}{{\mathbb{Q}}}
\newcommand{\X}{\mathbb{X}}
\newcommand{\R}{\mathbb{R}}

\newcommand{\Z}{\mathbb{Z}}

\newcommand{\cG}{\mathfrak{g}}

\newcommand{\cI}{{\mathscr{I}}}

\newcommand{\cL}{{\mathscr{L}}}

\newcommand{\cO}{{\mathscr{O}}}

\newcommand{\cU}{{\mathscr{U}}}

\newcommand\bpm{\begin{pmatrix}}
\newcommand\epm{\end{pmatrix}}

\newcommand\mhs{mixed Hodge structure}

\newcommand\hs{Hodge struc\-ture}

\newcommand\Pic{{\mathop{\rm Pic}\nolimits}}

\renewcommand\mod{\mathop{\rm mod}\nolimits}

 \newcommand{\Gr}{\mathop{\rm Gr}\nolimits}

 \newcommand\Aut{\mathop{\rm Aut}\nolimits}
\newcommand\End{\mathop{\rm End}\nolimits}

\newcommand\GL{\mathop{\rm GL}\nolimits}

\newcommand\Rim{\mathop{\rm Im}\nolimits}

\newcommand\Res{\mathop{\rm Res}\nolimits}

 \newcommand{\Alb}{{\rm Alb}}

\newcommand{\Ext}{\mathop{\rm Ext}\nolimits}

 \newcommand{\Hom}{\mathop{\rm Hom}\nolimits}
\newcommand{\Ker}{\mathop{\rm Ker}\nolimits}

 \renewcommand{\part}{\partial}
\newcommand{\la}{{\lambda}}
\newcommand\ga{{\gamma}}
\newcommand\Ga{{\Gamma}}
\newcommand{\La}{{\Lambda}}
\newcommand{\Om}{{\Omega}}
\newcommand{\om}{{\omega}}

\newcommand\vp{\varphi}
\newcommand\sig{\sigma}

\newcommand\bsl{\backslash}
\newcommand\eps{\epsilon}

\newcommand{\lab}{\label}

\newcommand{\hensp}[1]{\enspace\hbox{#1}\enspace}

\newcommand{\opplus}{\mathop{\oplus}\limits}
\newcommand{\ottimes}{\mathop{\otimes}\limits}

  \renewcommand\thesection{\Roman{section}}
 
\newcommand\bmp[2]{\hbox{\begin{minipage}[c]{#1in}  #2 \end{minipage}}}

\newcounter{demo}[equation]

 \newenvironment{demo}{
 \addtocounter{equation}{1}\refstepcounter{demo} \setcounter{demo}{\value{equation} }
{\smallbreak\noindent (\thesection.\arabic{demo})\enspace   }} 

 \newenvironment{ldemo}{
 \addtocounter{equation}{1}\refstepcounter{demo} \setcounter{demo}{\value{equation} }
{\smallbreak\noindent  \thesection.\arabic{demo} \enspace   }}

 \newtheoremstyle{mytheo}
  {3pt}
  {3pt}
  {\itshape}
  {}
  {\scshape}
  {:}
  {.5em}
  {}

\theoremstyle{mytheo}

\newtheorem{Quest}[equation]{Question}

\newtheorem*{thma}{Theorem A}
\newtheorem*{thmb}{Theorem B}
\newtheorem*{thmc}{Theorem C}

\newtheorem{Cor}[equation]{Corollary}

\newtheorem{Thm}[equation]{Theorem}

  \newtheorem{Conj}[equation]{Conjecture}

\newtheorem{Lem}[equation]{Lemma}

\newtheorem{Prop}[equation]{Proposition}

 \newtheoremstyle{subsect}
 {3pt}
  {3pt}
  {}
  {}
  {\it}
  {\upshape{:}}
  {.5em}
  {}
\theoremstyle{subsect}
\newtheorem{subsec}[equation]{}
\newtheoremstyle{note}
  {3pt}
  {3pt}
  {}
  {}
  {\bfseries}
  {:}
  {.5em}
  {}
\theoremstyle{note}

\newtheorem*{rem}{Remark}

\theoremstyle{remark}

\newcommand\xri[1]{\xrightarrow{#1}}

\newcommand\ol[1]{\overline{#1}}

 \usepackage[mathscr]{euscript}
\usepackage[color,matrix,arrow]{xy}
  \renewcommand\thesection{\Roman{section}}